\documentclass[a4paper]{amsart}
\usepackage[T1]{fontenc}
\usepackage{amsthm}
\usepackage{amssymb}
\usepackage[unicode=true,pdfusetitle,
 bookmarks=true,bookmarksnumbered=false,bookmarksopen=false,
 breaklinks=false,pdfborder={0 0 0},pdfborderstyle={},backref=false,colorlinks=false]
 {hyperref}
\usepackage{tikz-cd}
\usepackage{mathtools}
\usepackage{adjustbox}
\usepackage{enumitem}
\tikzcdset{scale cd/.style={every label/.append style={scale=#1},
    cells={nodes={scale=#1}}}}
\usepackage{quiver}
\usepackage{xcolor}
\usepackage{eucal}
\usepackage{tikz}
\usetikzlibrary{arrows.meta, positioning, shapes, fit}
\usepackage{hyperref}
\usepackage{mathrsfs}
\usepackage{stackrel}
\usepackage{mleftright}
\mleftright

\numberwithin{equation}{section}
\numberwithin{figure}{section}
\theoremstyle{plain}
\newtheorem{thm}[equation]{\protect\theoremname}
\theoremstyle{definition}

\theoremstyle{definition}
\newtheorem{defn}[equation]{\protect\definitionname}
\theoremstyle{definition}
\newtheorem{notation}[equation]{\protect\notationname}
\theoremstyle{remark}
\newtheorem{rem}[equation]{\protect\remarkname}
\theoremstyle{plain}
\newtheorem{prop}[equation]{\protect\propositionname}
\theoremstyle{plain}
\newtheorem{cor}[equation]{\protect\corollaryname}
\theoremstyle{plain}
\newtheorem*{conj}{\protect\conjecturename}

\providecommand{\conventionname}{Convention}
\providecommand{\conjecturename}{Conjecture}
\providecommand{\corollaryname}{Corollary}
\providecommand{\definitionname}{Definition}
\providecommand{\notationname}{Notation}
\providecommand{\propositionname}{Proposition}
\providecommand{\remarkname}{Remark}
\providecommand{\theoremname}{Theorem}
\newcommand{\notehelper}[3]{\textcolor{#3}{$\blacksquare$}\marginpar{\ifodd\thepage\raggedright\else\raggedleft\fi\color{#3}\tiny \textbf{#2:} #1}}

\theoremstyle{plain}
\newtheorem{customthm}{Theorem}

\theoremstyle{plain}
\newtheorem{customcon}{Convention}

\newcommand{\pr}[1]{\left(#1\right)}

\newcommand{\curlyCatinfty}{\mathcal{C}\mathsf{at}_{\infty}}
\newcommand{\curlyOpinfty}{\mathcal{O}\mathsf{p}_{\infty}}
\newcommand{\curlyRelCatinfty}{\mathcal{R}\mathsf{el}\mathcal{C}\mathsf{at}_{\infty}}
\newcommand{\curlyRelOpinfty}{\mathcal{R}\mathsf{el}\mathcal{O}\mathsf{p}_{\infty}}

\newcommand{\RelCatinfty}{\mathsf{RelCat}_{\infty}}
\newcommand{\RelOpinfty}{\mathsf{RelOp}_{\infty}}
\newcommand{\Catinfty}{\mathsf{Cat}_{\infty}}
\newcommand{\Opinfty}{\mathsf{Op}_{\infty}}

\newcommand{\QCat}{\mathsf{QCat}}
\newcommand{\Set}{\mathsf{Set}}
\newcommand{\sSet}{\mathsf{sSet}}
\newcommand{\dSet}{\mathsf{dSet}}
\newcommand{\CSS}{\mathsf{CSS}}
\newcommand{\Cat}{\mathsf{Cat}}
\newcommand{\DQOp}{\mathsf{DQOp}}
\newcommand{\LQOp}{\mathsf{LQOp}}
\newcommand{\CDSS}{\mathsf{CDSS}}
\newcommand{\Rel}{\mathsf{Rel}}
\newcommand{\Op}{\mathsf{Op}}
\newcommand{\RelOp}{\mathsf{RelOp}}

\newcommand{\Lurie}{\mathrm{Lurie}}
\newcommand{\Rezk}{\mathrm{Rezk}}
\newcommand{\hc}{\mathrm{hc}}
\newcommand{\id}{\mathrm{id}}

\newcommand{\Kan}{\mathrm{Kan}}

\newcommand{\Hom}{\operatorname{Hom}}
\newcommand{\Fun}{\operatorname{Fun}}
\newcommand{\Map}{\operatorname{Map}}
\newcommand{\qAlg}{\operatorname{Alg}}
\newcommand{\Alg}{\mathsf{Alg}}

\newcommand{\Del}{\mathbf{\Delta}}
\newcommand{\Den}{\mathbf{\Omega}}

\newcommand{\op}{\mathrm{op}}
\newcommand{\from}{\colon}

\begin{document}
\title{Relative operads model $\infty$-operads}
\author{K.\ Arakawa, V.\ Carmona, and F.\ Pratali}
\keywords{relative $\infty$-operads, localization, higher algebra, operadic nerve}
\subjclass{18Nxx, 18N55, 18N70, 18M60, 18M75, 55U35}
\begin{abstract}
Given a (colored) operad and a set of unary operations, we can
form an associated $\infty$-operad via localization.
We show that localization determines an equivalence of homotopy
theories of relative operads and $\infty$-operads. As an application,
we give an affirmative answer to an open question by
Harpaz, proving that Lurie's operadic nerve functor determines
an equivalence of homotopy theories of simplicial operads and
Lurie's $\infty$-operads.
\end{abstract}

\maketitle

%\textcolor{blue}{To keep or erase some of the previous items (add more if you like)
%	\begin{itemize}
	%	\item 18Nxx: higher cats and homotopical algebra
	%	\item 18N55: Localizations (e.g.\ simplicial, Bousfield...)
	%	\item 18N70: $\infty$-operads and higher algebra
	%    \item 18M60: operads (general)
	%    \item 18M75: topological and simplicial operads 
	%	\item 55U35: Abstract and axiomatic homotopy theory in algebraic topology
%\end{itemize}}

\section*{Introduction}
\textbf{Relative categories} are a remarkably simple yet powerful
model of $\pr{\infty,1}$-categories, or $\infty$-categories
for short. A relative category is a pair $\pr{\mathcal{C},W}$,
where $\mathcal{C}$ is a category and $W\subseteq \mathcal{C}$ is a replete subcategory, whose morphisms are called weak
equivalences.
\footnote{In the original work of Barwick and Kan \cite{BK12a, BK12b}, $W$ is only required to contain all objects.
We modify the definition for this paper because we do not use wideness anywhere, and we can only speak of replete subcategories when working invariantly with $\infty$-categories. 
The modification does not affect the homotopy theory of relative categories.}

The definition formalizes the common situation in mathematics
where a category has non-invertible morphisms that ``behave like'' isomorphisms. To treat them as such, 
one freely inverts maps in $W$ to obtain an $\infty$-category
$\mathcal{C}[W^{-1}]$, called the \textbf{localization} of $\mathcal{C}$ at $W$. From
the early days of homotopy theory, it has been observed that
many $\infty$-categories can be presented as a localization
of a relative category \cite{Quillen_HA, DK80_1, DK80_2, DK80_3}.
Such presentations are extremely useful, as they allow various
$\infty$-categorical constructions to be performed at the level
of ordinary categories, where they are considerably more tractable.

It is natural to ask if such a presentation always exists. Joyal's
\textit{delocalization theorem}, first stated by Joyal \cite[\S 13.6]{J:NQC} and then proven by Stevenson \cite[Theorem 1.3]{Ste:CMSSL}, 
%\cite[Theorem 3.3.8]{Landoo-cat}
gives an affirmative answer to this question, saying that \textit{every}
$\infty$-category is a localization of a relative category.
In fact, by the work of Barwick--Kan \cite{BK12a, BK12b}, we know that 
localization determines an equivalence of $\infty$-categories
\[
\mathsf{RelCat}[\mathrm{loc.eq}^{-1}]\xrightarrow{\sim}\curlyCatinfty,
\]
where $\mathrm{loc.eq}$ denotes the subcategory of \textit{local
equivalences}, i.e., maps of relative categories inducing an equivalence
between localizations.\\

Localizing categories is a well-studied procedure, but there is a growing interest in extending localization techniques to other category-like structures that encode additional algebraic information, such as \textbf{operads} (also known as symmetric multicategories). Operads can be viewed as a generalization of categories in which morphisms are allowed to have multiple inputs, thereby capturing multiplicative or compositional structures beyond those present in ordinary categories. If we have
an operad, we can freely invert a collection of its unary arrows
to obtain an $\infty$-operad, called its \textbf{localization}. More precisely, defining a \textbf{relative operad} as
a pair $\pr{\mathcal{O},S}$, where $\mathcal{O}$ is an operad
and $S$ is a replete subcategory of the category of unary operations $U(\mathcal{O})$ in $\mathcal{O}$,  its \textbf{localization}
is the $\infty$-operad $\mathcal{O}[S^{-1}]$ obtained by
freely inverting the operations in $S$.

This construction has gained some attention in recent times from both a theoretical perspective (e.g., \cite{BBPTY18, Pra25}) and a practical perspective, especially in the context of \textit{factorization algebras} and \textit{algebraic quantum field theories}
in mathematical physics, where it is used to handle topological and dynamical axioms (e.g., \cite{benini_equivalence_2024,benini_cast-categorical_2025,benini_strictification_2023,benini_operads_2025,calaque_not_2024, CC_notlittle, carmona_additivity_2025, HA,karlsson_assembly_2025}). 

In \cite{Pra25}, the third author extended Joyal's delocalization
theorem to the operadic setting, showing that every $\infty$-operad
is a localization of an ordinary (discrete) operad. %$\Den/\mathcal{O}$ via a direct map $\Den/\mathcal{O}\to \mathcal{O}$. 
%More precisely, the third author defined, for every $\infty$-operad $\mathcal{O}$, realized as a dendroidal quasioperad, an ordinary operad$\Omega/\mathcal{O}$ equipped with a map
%\[\rho_{\mathcal{O}}:\Omega/\mathcal{O}\to\mathcal{O}\]
%called the \textbf{root functor}. The root functor is naturalin $\mathcal{O}\in\DQOp$, and moreover if $\mathcal{O}$ iscofibrant in the operadic model structure, then its image in$\curlyOpinfty$ is a localization at the unary operations whose
%<<<<<<< Francesca
%images in $\mathcal{O}$ are equivalences \cite[Theorem 3.13]{Pra25}. \\
%=======
%images in $\mathcal{O}$ are equivalences \cite[Theorem 3.13]{Pra25}.
% \ken{I understand that you want to advertise your work (you're the inventor!), but I think it's also important to get to the main statement quickly without wasting the reader's mental resources. I'd prefer the previous version.}
%>>>>>>> master
In this paper, we use this result to
obtain an equivalence of homotopy theories, thereby
establishing an operadic version of Barwick--Kan's theorem: 

%To state our main result, define a \textbf{relative operad} as
%a pair $\pr{\mathcal{O},S}$, where $\mathcal{O}$ is an operad
%and $S$ is a (replete) subcategory of the category of unary operations $U(\mathcal{O})$ in $\mathcal{O}$. We also define its \textbf{localization}
%to be the $\infty$-operad $\mathcal{O}[S^{-1}]$ obtained by
%freely inverting the operations in $S$. With these definitions,
%we show that:
\begin{customthm}
[Theorem \ref{thm:main}] \label{thm:main_intro} Localization of relative operads
determines an equivalence of $\infty$-categories
\[
\RelOp[\mathrm{loc.eq}^{-1}]\xrightarrow{\sim}\curlyOpinfty.
\]
\end{customthm}

As an application of Theorem \ref{thm:main_intro}, we solve
an open problem raised by Harpaz on MathOverflow \cite{MO249973}.
%The problem is about Lurie's\textit{ operadic nerve functor},
%which converts locally Kan simplicial operads into Lurie's $\infty$-operads.
The question is whether Lurie's\textit{ operadic nerve functor}, $N_{\Lurie}\from\Op_{\Kan}\to\LQOp$ from locally Kan simplicial operads
to Lurie quasioperads, determines an equivalence between the associated $\infty$-categories. %\francesca{added} So far, the conjecture is known to be true only when one restrict $N_{\Lurie}$ on both sides to operads \emph{without units}, also called \emph{non-unital} or \emph{open} operads, by \cite[Proposition 6.1.1]{HHM:ELMDMIO}. However, beyond the fact that their solution is only partial, their method is rather involved, as it requires comparing the restricted functor to a chain of derived functors between homotopy categories which are known to be equivalences.

While the equivalence is expected to be ``very likely'' for the reason that it is ``quite a natural
construction'' between two models of $\infty$-operads \cite{MO249973}, the comparison has proved to be difficult because
model-categorical techniques do not directly apply. (The operadic
nerve functor is neither left nor right Quillen.)

However, with Theorem \ref{thm:main_intro} in hand, we can give a clear proof of the full conjecture, confirming Harpaz's expectation:
\begin{customthm}
[Theorem \ref{thm:Harpaz}] \label{thm:Harpaz_intro} The operadic
nerve functor determines an equivalence of $\infty$-categories  $N_{\Lurie}\from\Op_{\Kan}[\mathrm{weq}^{-1}]\xrightarrow{\sim}\LQOp[\mathrm{weq}^{-1}]$.
\end{customthm}

Roughly speaking, Theorem \ref{thm:main_intro}
facilitates the comparison because relative operads and their
localization can be interpreted in any reasonable model of $\infty$-operads.
As such, we expect that our method extends to the comparison
of various other models of $\infty$-operads.

\subsection*{\label{subsec:relation}Relation with other works}
Whereas Barwick--Kan's equivalence is established on the level of model categories, our equivalence for operads is formulated at the level of relative categories. While it should be possible in principle to extend Barwick–Kan’s model categories to the operadic setting, the added technical complexity offers little benefit, so we do not pursue this here. \\

A partial positive answer to Harpaz's conjecture was given, even before its formulation, by Heuts-Hinich-Moerdijk in \cite[Proposition 6.1.1]{HHM:ELMDMIO}, which establishes the equivalence for operads without units (i.e.\ \emph{open} operads). Theorem \ref{thm:Harpaz_intro} recovers and extends the equivalence in \emph{loc.\ cit.}. That paper also appears to be the first to formulate Lurie's operadic nerve as a functor between relative categories---a perspective we likewise adopt. Beyond this, the approaches differ: the equivalence for open operads in \emph{loc.\ cit.} is obtained via a zig-zag of Quillen equivalences, whereas our proof relies directly on Theorem \ref{thm:main_intro} and is perhaps more conceptual.

\subsection*{\label{subsec:organization}Organization of the paper}

This paper consists of three sections. In Section \ref{sec:prelims},
we recall basic definitions and facts on $\infty$-operads and
$\infty$-categories. In Section \ref{sec:main}, we prove Theorem
\ref{thm:main_intro} after a brief discussion of relative $\infty$-operads
and their localization. In Section \ref{sec:appl}, we characterize localization of simplicial and Lurie's $\infty$-operads by a Morita-type universal property (Propositions \ref{prop:char_loc_Kan} and \ref{prop:char_loc_Lurie}) and prove
Theorem \ref{thm:Harpaz_intro}.

We also adopt the following convention:
\begin{customcon}
\label{conv:oo-cats}Throughout the paper, we make extensive
use of \textbf{$\infty$-categories}. In particular, unless stated
otherwise, various categorical constructions (colimits, localizations,
etc.) are performed in the realm of $\infty$-categories. For
the most part, readers can ignore which model of $\infty$-categories
we use. But when pressed, we mean quasicategories in the sense
of Joyal \cite{Joyal_qcat_Kan}, which is the most well-developed
model of $\pr{\infty,1}$-categories. We will not notationally
distinguish between categories and their nerves, and we use the
term ``\textbf{$\infty$-groupoid}'' as a synonym of ``Kan
complex''. 
\end{customcon}

\section{\label{sec:prelims}Preliminaries}

In this section, we recall some basic definitions and facts about
$\infty$-operads and $\infty$-categories.

\subsection{\label{subsec:models_oo-cats}Models of $\infty$-categories}

There are many equivalent models of $\infty$-categories. For
the purpose of this paper, we will be concerned with the following
three models:
\begin{itemize}
\item The category $\QCat$ of \textbf{quasicategories}, which are
fibrant objects of the Joyal model structure on the category
of simplicial sets \cite[Theorem 2.2.5.1]{HTT}. 
%Its morphisms are maps of simplicial sets.
\item The category $\CSS$ of \textbf{complete Segal spaces}, which
are fibrant objects of the complete Segal space model structure
on the category of simplicial simplicial sets (i.e., bisimplicial
sets) \cite[Theorem 7.2]{Rezk01}. 
% Its morphisms are the maps of bisimplicial sets.
\item The category $\Cat_{\Kan}$ of \textbf{locally Kan simplicial
categories}, which are fibrant objects of the Bergner model structure
on simplicial categories \cite{Ber07}.
% Its morphisms are the maps of simplicial categories.
\end{itemize}

We will think of these categories as relative categories whose weak equivalences are the weak equivalences of the respective model categories.

These relative categories are related by the
functors 
\[
\QCat\xrightarrow{N_{\Rezk}}\mathsf{CSS}\xrightarrow{q}\QCat\xleftarrow{N_{\hc}}\Cat_{\Kan},
\]
where:
\begin{itemize}
\item $N_{\Rezk}$ denotes the \textbf{Rezk nerve }functor, given by
\[
N_{\Rezk}\pr{\mathcal{C}}=\Map\pr{\Delta^{\bullet},\mathcal{C}}.
\]
Here $\Map$ denotes the maximal sub Kan complex of the internal
hom of simplicial sets.
\item $q$ denotes the restriction of the functor $\mathrm{ev}_{0}^{\Del^{\op}}\from\sSet^{\Del^{\op}}\to\mathsf{Set}^{\Del^{\op}}$.
\item $N_{\hc}$ denotes the \textbf{homotopy coherent nerve} functor
\cite{Cordier_hcnerve}.
\end{itemize}
The following theorem asserts that these functors induce equivalences
of homotopy theories:
\begin{thm}
\label{thm:models_oo-cats}The functors
\[
\QCat\xrightarrow{N_{\Rezk}}\CSS\xrightarrow{q}\QCat\xleftarrow{N_{\hc}}\Cat_{\Kan}
\]
induce equivalences of $\infty$-categories upon localizing at
weak equivalences.
\end{thm}

\begin{proof}
The claim on $q$ is proved in \cite[Theorem 4.11]{JT07}, and
the claim on $N_{\Rezk}$ follows from \cite[Proposition 1.20 and Theorem 4.12]{JT07}.
The claim on $N_{\hc}$ is proved in \cite[Theorem 2.2.5.1]{HTT}.
\end{proof}
Because of Theorem \ref{thm:models_oo-cats}, we make the following
definition:
\begin{defn}
\label{def:oo-cats}We will write $\mathsf{Cat}_{\infty}$ for
any one of the relative categories $\QCat$, $\CSS$, or $\Cat_{\Kan}$,
and refer to its objects as \textbf{concrete $\infty$-categories}.
We write $\curlyCatinfty$ for the localization $\mathsf{Cat}_{\infty}[\mathrm{weq}^{-1}]$
and call its objects \textbf{$\infty$-categories}.
\end{defn}

\begin{rem}
Definition \ref{def:oo-cats} does not conflict Convention \ref{conv:oo-cats},
because we can choose a localization functor $\QCat\to\QCat[\mathrm{weq}^{-1}]$
that does not change the collection of objects.
\end{rem}

\subsection{\label{subsec:models_oo-opds}Models of $\infty$-Operads}

Each of the models of $\infty$-categories we discussed in Subsection
\ref{subsec:models_oo-cats} has operadic generalizations:
\begin{itemize}
\item The category $\DQOp$ of \textbf{dendroidal quasioperads}, which
are fibrant objects of the operadic model structure on the category
of dendroidal sets \cite[Theorem 2.4]{CM11}.
% Its morphisms are the maps of dendroidal sets. 
\item The category $\CDSS$ of \textbf{complete dendroidal Segal spaces},
which are fibrant objects of the operadic model structure on
the category of dendroidal simplicial sets \cite[Definition 6.2]{CM13a}.
%Its morphisms are the maps of dendroidal simplicial sets.
\item The category $\Op_{\Kan}$ of \textbf{locally Kan simplicial
operads}, which are the fibrant objects of the Dwyer--Kan model
structure on the category of simplicial operads \cite[Theorem 1.14]{CM13b}.
%Its morphisms are the maps of simplicial operads.
%Its morphisms are the maps of $\infty$-preoperads.
\end{itemize}
There is yet another model for $\infty$-operads, perhaps the most developed so far, which \emph{builds on}, instead of \emph{generalizing}, the theory of $\infty$-categories:
\begin{itemize}
\item The category $\LQOp$ of \textbf{Lurie quasioperads}, which are
the fibrant objects of the preoperadic model structure on the category of $\infty$-preoperads \cite[Proposition 2.1.4.6]{HA}.
\end{itemize}

Again, we will see each of these models as a relative category, where weak equivalences are the weak equivalences in the respective model category. These categories and the previously discussed
models of $\infty$-categories fit into the following diagram,
which commutes up to natural weak equivalence:
\begin{equation}\label{d:models}
\begin{tikzcd}
	{\mathsf{LQOp}} & {\mathsf{CDSS}} & {\mathsf{DQOp}} & {\mathsf{Op}_{\mathrm{Kan}}} \\
	{\mathsf{QCat}} & {\mathsf{CSS}} & {\mathsf{QCat}} & {\mathsf{Cat}_{\mathrm{Kan}}.}
	\arrow["\delta", from=1-1, to=1-2]
	\arrow["{U_1}"', from=1-1, to=2-1]
	\arrow["{(1)}"{description}, draw=none, from=1-1, to=2-2]
	\arrow["q", from=1-2, to=1-3]
	\arrow["{U_2}"', from=1-2, to=2-2]
	\arrow["{(2)}"{description}, draw=none, from=1-2, to=2-3]
	\arrow["{U_3}"', from=1-3, to=2-3]
	\arrow["{(3)}"{description}, draw=none, from=1-3, to=2-4]
	\arrow["{N_{d}}"', from=1-4, to=1-3]
	\arrow["{U_4}"', from=1-4, to=2-4]
	\arrow["{N_{\mathrm{Rezk}}}"', from=2-1, to=2-2]
	\arrow["q"', from=2-2, to=2-3]
	\arrow["{N_{\mathrm{hc}}}", from=2-4, to=2-3]
\end{tikzcd}
\end{equation}Here is the explanation of the functors in the diagram:
\begin{itemize}
\item $\delta$ is given by $\delta\pr{\mathcal{O}}=\Map_{\LQOp}\pr{i\pr -,\mathcal{O}}$.
Here $\Map_{\LQOp}$ denotes the mapping space for the simplicial
model structure for $\infty$-preoperads, and the functor  $i\from\Den\subset \Op\to\LQOp$
denotes the inclusion given by Lurie's operadic nerve functor
\cite[Definition 2.1.1.23]{HA}.
\item $q$ denotes the restriction of the functor $\mathrm{ev}_{0}^{\Den^{\op}}\from\sSet^{\Den^{\op}}\to\Set^{\Den^{\op}}$, right adjoint to the inclusion;
\item $N_{d}$ denotes the \textbf{dendroidal homotopy coherent nerve}
functor \cite[\textsection 7]{MW09} ($hcN_{d}$ in loc.cit.).
\item The functors $U_{i}$'s extract the $\infty$-category of
unary operations. More precisely, $U_{1}$ takes the fiber over
$\langle1\rangle\in\mathsf{Fin}_{\ast}$, $U_{2}$ and $U_3$ are induced
by the inclusion $\Del\hookrightarrow\Den$, and $U_{4}$ sends a
simplicial operad to its simplicial category of unary operations.
\end{itemize}
Square (1) commutes up to natural weak equivalence by \cite[Example 2.1.3.5]{HA};
square (2) and (3) commute strictly.

The category $\Op$ of operads embeds fully faithfully into each
of these categories, in the following ways:

\begin{enumerate}[label=(\roman*)]
\item using the dendroidal nerve functor $N_d\colon \Op\to \DQOp$ \cite[\textsection 2]{MW09};
%\item the inclusion $\DQOp \subseteq \CDSS$ given by regarding a dendroidal set as a discrete dendroidal space comes from a left Quillen equivalence of model categories (\cite[Corollary 6.7]{CM13a}); one gets a functor $\Op\to \CDSS$ by postcomposing the dendroidal nerve in (i) with the above inclusion followed by any functor implementing the Segal completion\footnote{We will not need this in this paper.};\victor{Should we say something about $\CDSS$?} \francesca{what about this?}
\item via the inclusion $\Op\to\Op_{\Kan}$
which regards an operad as a simplicial operad with discrete mapping
objects,
\item using Lurie's operadic nerve functor 
 $N_{\Lurie}\from\Op\to\LQOp$ \cite[Definition 2.1.1.23]{HA} and the functor $\delta\from \LQOp \to \CDSS$. %\ken{Thanks Francesca for rewriting this part. I made a further edit for the embedding into $\CDSS,$ which I think is more canonical.}
\end{enumerate}
%(i) using the dendroidal nerve functor $N_d\colon \Op\to \DQOp$ \cite[\textsection 2]{MW09}; (ii) via the inclusion $\Op\to\Op_{\Kan}$which regards an operad as a simplicial operad with discrete mappingobjects; and (iii) composing the last inclusion with Lurie's operadic nerve functor$N_{\Lurie}:\Op_{\Kan}\to\LQOp$ \cite[Definition 2.1.1.23]{HA}. We will not notationally distinguish between operads with their
%images in the categories $\DQOp$, $\CDSS$, $\Op_{\Kan}$, and $\LQOp$.

The following result is an operadic analog of Theorem \ref{thm:models_oo-cats}.
\begin{thm}
\label{thm:models_oo-operads}The functors 
\[
\LQOp\xrightarrow{\delta}\CDSS\xrightarrow{q}\DQOp\xleftarrow{N_{d}}\Op_{\Kan}
\]
induce equivalences of $\infty$-categories upon localizing at
weak equivalences.
\end{thm}

\begin{proof}
For $\delta$, the claim is proved in \cite[Theorem 1.1.1]{HM24};
for $q$, this is \cite[Corollary 6.7]{CM13a}; and for $N_{d}$,
this is \cite[Theorem 8.15]{CM13b}.
\end{proof}
With Theorem \ref{thm:models_oo-operads} in mind, it will be
useful to make use of the following definition:
\begin{defn}
We write $\Op_{\infty}$ for any of the relative categories
$\LQOp,\CDSS,\DQOp$, or $\Op_{\Kan}$, and write $\curlyOpinfty$
for its localization at weak equivalences. We refer to the objects of $\Op_{\infty}$
as \textbf{concrete $\infty$-operads} and to the objects of $\curlyOpinfty$
as \textbf{$\infty$-operads}. \\

Diagram \ref{d:models} gives us forgetful functors $U\from\Opinfty\to\Catinfty$
and $U\from\curlyOpinfty\to\curlyCatinfty$. If $\mathcal{O}$
is a concrete $\infty$-operad (resp.\ an $\infty$-operad), we refer
to $U\pr{\mathcal{O}}$ as the \textbf{concrete} $\infty$-\textbf{category of unary
operations} (resp.\ the \textbf{$\infty$-category of unary operations}).
\end{defn}

\begin{rem}
\label{rem:operadification}The functor $U\from\curlyOpinfty\to\curlyCatinfty$
admits a fully faithful left adjoint. (This is because, for example,
the functor $U_{3}:\Op_{\Kan}\to\Cat_{\Kan}$ is part of a Quillen
adjunction whose derived unit is an isomorphism.) We will identify
$\curlyCatinfty$ as a full subcategory of $\curlyOpinfty$ using
this left adjoint.
\end{rem}

\begin{rem}\label{rem:symmetric monoidal structures on Op infty}The $\infty$-category $\curlyOpinfty$ has a symmetric monoidal structure. This can be seen for instance by observing that Lurie's model structure for quasioperads is a symmetric monoidal model category (\cite[Proposition 2.2.5.7]{HA}), or by using Hinich-Moerdijk's result lifting the tensor product of dendroidal sets to a symmetric monoidal structure on the $\infty$-category $\CDSS$ (\cite[\S 5]{HM24}). That these two tensor products are equivalent is ensured by \cite[Theorem 1.1.1]{HM24}.
	%Some (but not all) models of $\infty$-operads are parts of symmetric monoidal model categories (for example, this is true for  Lurie quasioperads \cite[Proposition 2.2.5.7]{HA}).
	%It follows that the $\infty$-category $\curlyOpinfty$ inherits a symmetric monoidal structure.
	Unravelling the definitions, we find that the inclusion $\curlyCatinfty \hookrightarrow \curlyOpinfty$ can be enhanced to a symmetric monoidal functor, where $\curlyCatinfty$ carries the cartesian monoidal structure.
	It follows that $\curlyOpinfty$ can be enriched over $\curlyCatinfty$ (that is, it can be enhanced to an $(\infty,2)$-category) in the sense of \cite[Definition 4.2.1.28]{HA}.
	We denote its mapping categories by $\qAlg_{-}(-)$. Thus we have an equivalence
	\[
		\Map_{\curlyOpinfty}(\mathcal{C}\otimes \mathcal{O},\mathcal{P}) \simeq \Map_{\curlyCatinfty}(\mathcal{C},\qAlg_{\mathcal{O}}(\mathcal{P}))
	\]
natural in all variables.
\end{rem}

\section{\label{sec:main}Main Result}

In this section, we state and prove the main result of this paper,
asserting that relative operads model $\infty$-operads (Theorem
\ref{thm:main}). To state our main result precisely, we start
with a brief discussion on relative $\infty$-categories and
relative $\infty$-operads.

\subsection{Relative $\infty$-categories}

In this subsection, we introduce relative $\infty$-categories.
\begin{defn}
A \textbf{concrete relative $\infty$-category} is a pair $\pr{\mathcal{C},S}$,
where $\mathcal{C}$ is a concrete $\infty$-category and $S$
is a set of morphisms of $\mathcal{C}$ which is stable under compositions of $\mathcal{C}$, and which is stable under natural equivalences when regarded as a set of functors $[1]\to \mathcal{C}$ in $\curlyCatinfty$.
%\ken{I've modified the definition here in light of Victor's comment.}

A \textbf{relative functor} $\pr{\mathcal{C},S}\to\pr{\mathcal{D},T}$
is a map $\mathcal{C}\to\mathcal{D}$ of concrete $\infty$-categories
carrying $S$ into $T$. We let $\RelCatinfty$ denote the category
of relative concrete $\infty$-categories and relative functors.
\end{defn}

\begin{defn}
A \textbf{relative $\infty$-category} is a monomorphism $S\to\mathcal{C}$
of $\curlyCatinfty$, which we often abbreviate as $\pr{\mathcal{C},S}$.
We let $\curlyRelCatinfty\subset\Fun\pr{[1],\curlyCatinfty}$
denote the full subcategory spanned by relative $\infty$-categories.
\end{defn}

\begin{rem}
	\label{rem:rel_oo-cats_mappingspace}Using the description of
	the mapping spaces of the arrow category of $\infty$-categories
	in \cite[Proposition B.1]{Ara_calc}, we obtain the following
	description of the mapping space of $\curlyRelCatinfty$: If
	$\pr{\mathcal{C},S}$ and $\pr{\mathcal{D},T}$ are relative
	$\infty$-categories, then the map
	\[
	\Map_{\curlyRelCatinfty}\pr{\pr{\mathcal{C},S},\pr{\mathcal{D},T}}\to\Map_{\curlyCatinfty}\pr{\mathcal{C},\mathcal{D}}
	\]
	is fully faithful, with essential image consisting of those maps
	$\mathcal{C}\to\mathcal{D}$ whose restriction to $S$ factors
	through $T$.
\end{rem}

An important operation associated with relative $\infty$-categories is \textit{localization}:

\begin{defn}\label{def:loccats}
	Let $\pr{\mathcal{C},S}$ be a relative $\infty$-category.
	We say that a map $\gamma:\mathcal{C}\to\mathcal{D}$ of $\infty$-categories
	\textbf{exhibits $\mathcal{D}$ as a localization of $\mathcal{C}$
		at $S$ }if for every $\infty$-category $\mathcal{Z}$, precomposition
	with $\gamma$ induces an equivalence
	\[
	\Map_{\curlyCatinfty}\pr{\mathcal{D},\mathcal{Z}}\xrightarrow{\sim}\Map_{\curlyCatinfty}^{S}\pr{\mathcal{C},\mathcal{Z}},
	\]
	where $\Map_{\curlyCatinfty}^{S}\pr{\mathcal{C},\mathcal{Z}}\subset\Map_{\curlyCatinfty}\pr{\mathcal{C},\mathcal{Z}}$
	denotes the full sub $\infty$-groupoid spanned by the maps $\mathcal{C}\to\mathcal{Z}$
	carrying each morphism of $S$ into an equivalence of $\mathcal{Z}$.
	If this condition is satisfied, we write $\mathcal{D}=\mathcal{C}[S^{-1}]$.
\end{defn}

% \francesca{added Def \ref{def:loccats} and Rmk \ref{rem:loc_cat}}Let us recall here the notion of localization of $\infty$-categories, which will be central for stating our results.

% Localization always exists and can be taken to be functorial in the relative $\infty$-categories. In other words, localization defines a left adjoint to the inclusion functor $$\curlyCatinfty\to \curlyRelCatinfty , \quad \mathcal{C}\mapsto \pr{\mathcal{C},\text{equivalences}(\mathcal{C})}$$ sending an $\infty$-category to the 'minimal' relative category on it. If we call $L$ the left adjoint, then for every relative categroy $\pr{\mathcal{C},S}$, the localization map $\gamma\colon \mathcal{C}\to\mathcal{C}[S^{-1}]$ can be taken as the unit of the adjunction at $\pr{\mathcal{C},S}$ seen as a morphism in $\curlyCatinfty$ (cfr Remark \ref{rem:rel_oo-cats_mappingspace}.)
\begin{rem}
	Note that the localization of a relative $\infty$-category ($\mathcal{C},S$) always exists; for instance, we can just take the pushout $S[S^{-1}]\amalg_{S}\mathcal{C}$, where $S[S^{-1}]$ denotes the $\infty$-groupoid completion of $S$. 
It follows from Remark \ref{rem:rel_oo-cats_mappingspace} that localization can even be constructed functorially: The inclusion
\[
	\curlyCatinfty \hookrightarrow \curlyRelCatinfty,\, \mathcal{C}\mapsto (\mathcal{C},\mathcal{C}^{\simeq}) = (\mathcal{C}, \text{equivalences})
\]
admits a left adjoint, with unit map given by localization. 
\end{rem}
\begin{rem}\label{rem:loc_cat}
	% \cite[Remark 1.19]{MG:URN}\ken{Mazel-Gee's proof is much nicer than Kerodon. Thanks! I also omitted the citation to HA, as Lurie does not use the word ``localization'' there.} \victor{MG's proof is essentially your argument the other day. Should we say that it follows from the fact that  $Fun(-,\mathcal{Z})$ is cocontinuous? It's a one line reason.}
Localization of $\infty$-categories enjoys an $(\infty,2)$-universal property.
More precisely, a functor $\gamma\from \mathcal{C}\to \mathcal{D}$ of $\infty$-categories exhibits $\mathcal{D}$ as a localization of $\mathcal{C}$ at $S$ if and only if it satisfies the following apparently stronger condition: 
For every $\infty$-category $\mathcal{Z}$, precomposition with $\gamma$ induces an equivalence of $\infty$-categories
	\[
		\Fun(\mathcal{D},\mathcal{Z})\xrightarrow{\sim}\Fun^S(\mathcal{C},\mathcal{Z}),
	\]
where $\Fun(-,-)$ denotes the internal hom of the cartesian monoidal $\infty$-category $\curlyCatinfty$, and the right-hand side is the full subcategory of $\Fun(\mathcal{C},\mathcal{Z})$ spanned by the functors carrying each map in $S$ to equivalences in $\mathcal{Z}$. This universal property follows from the Yoneda lemma, as we have the identification
\[
	\Map_{\curlyCatinfty}(\mathcal{A},\Fun^S(\mathcal{C},\mathcal{Z})) \simeq 
	\Map_{\curlyCatinfty}^S(\mathcal{C},\Fun(\mathcal{A},\mathcal{Z})).
\]
% This latter is in fact the definition of localization given by Lurie in \cite[Definition 1.3.4.1.]{HA}. 

Beware that in some literature (most notably \cite{HTT}), the term ``localization'' is reserved for functors admitting fully faithful right adjoints. Because of this, localization in our sense is sometimes called \textit{Dwyer--Kan} localization.
% The natural setting for defining localization of relative categories is as a left adjoint to the inclusion functor sending an $\infty$-category $\mathcal{C}$ to the 'minimal' relative $\infty$-category on it. This definition captures both the universal property in Definition \ref{def:loc} with respect to mapping spaces and functoriality in the relative operad. This is what is usually refered to as \emph{left Bousfield localization}. 
	% There exists also another notion of localization of $\infty$-categories, sometimes refered to under the name \emph{Dwyer-Kan localization}: this is for instance Lurie's notion of localization of quasicategories given in \cite[Definition 1.3.4.1.]{HA}, which exploits the self-enrichment of $\Catinfty$ and prescribes that the universal property of Definition \ref{def:loccatsl} holds at the level of the $\infty$-categories of functors. 
	%
	% These two notions are different, and in particular a Dwyer-Kan localization does not need to be a left Bousfield localization. However, when this latter exists, then it recovers the Dwyer-Kan localization, see for instance \cite[Remark 1.19]{MG:URN}.
\end{rem}

We say that a relative functor of concrete relative $\infty$-categories
is a \textbf{relative equivalence} if its image in $\curlyRelCatinfty$
is an equivalence. With this definition, we have:
\begin{prop}
\label{prop:deloc_cat}The functor
\[
\RelCatinfty[\mathrm{rel.eq}^{-1}]\to\curlyRelCatinfty
\]
 is an equivalence of $\infty$-categories, where $\mathrm{rel.eq}$ denotes the subcategory of maps whose
 image in $\curlyRelCatinfty$ is an equivalence.
\end{prop}

\begin{proof}
The functor under consideration is essentially surjective, so
it suffices to show that it is fully faithful. We prove this
by using the derived mapping space lemma \cite[Theorem 2.2]{ACK25},
which identifies the mapping spaces of a localization using resolutions. 

Given a concrete relative $\infty$-category $\pr{\mathcal{C},S}$,
choose a cosimplicial resolution $\mathcal{C}^{\bullet}$ of
$\mathcal{C}$ in whichever model structure defining $\Catinfty$.
We extend $\mathcal{C}^{\bullet}$ to a cosimplicial object $\pr{\mathcal{C}^{\bullet},S^{\bullet}}$
in $\RelCatinfty$ so that for each $n\geq0$, the map $\pr{\mathcal{C}^{n},S^{n}}\to\pr{\mathcal{C},S}$
is a relative equivalence.
(Explicitly, $S^n$ is the preimage of $S$ under $\mathcal{C}^n\to \mathcal{{C}}$.)
% \victor{It is crystal clear now :)}
If $\pr{\mathcal{D},T}$ is any other
concrete relative $\infty$-category, then the simplicial set
$\RelCatinfty\pr{\pr{\mathcal{C}^{\bullet},S^{\bullet}},\pr{\mathcal{D},T}}$
is the full sub $\infty$-groupoid of $\Catinfty\pr{\mathcal{C}^{\bullet},\mathcal{D}}$
spanned by the vertices $\mathcal{C}^{0}\to\mathcal{D}$ determining
a morphism $\pr{\mathcal{C}^{0},S^{0}}\to\pr{\mathcal{D},T}$
of concrete relative $\infty$-categories. 

According to \cite[Theorem 2.2]{ACK25}, the map 
$$
\Hom_{\RelCatinfty}\pr{\pr{\mathcal{C},S},-}\to\Hom_{\RelCatinfty}\pr{\pr{\mathcal{C}^{\bullet},S^{\bullet}},-}
$$
of functors $\RelCatinfty\to\mathcal{S}$ induces an equivalence
\[
\Map_{\RelCatinfty[\mathrm{rel.eq}^{-1}]}\pr{\pr{\mathcal{C},S},-}\xrightarrow{\sim}\Hom_{\RelCatinfty}\pr{\pr{\mathcal{C}^{\bullet},S^{\bullet}},-}.
\]
Similarly, we
have an equivalence
\[
\Map_{\curlyCatinfty}\pr{\mathcal{C},-}\xrightarrow{\sim}\Hom_{\Catinfty}\pr{\mathcal{C}^{\bullet},-}.
\]
Combining them, we find that for any concrete relative
$\infty$-operad $\pr{\mathcal{D},T}$, the map
\[
\Map_{\RelCatinfty[\mathrm{rel.eq}^{-1}]}\pr{\pr{\mathcal{C},S},\pr{\mathcal{D},T}}\to\Map_{\curlyCatinfty}\pr{\mathcal{C},\mathcal{D}}
\]
is a monomorphism of $\infty$-groupoids, with essential image
consisting of those maps $\mathcal{C}\to\mathcal{D}$ lifting
to a morphism of relative $\infty$-categories $\pr{\mathcal{C},S}\to\pr{\mathcal{D},T}$.
It follows from Remark \ref{rem:rel_oo-cats_mappingspace} the
map $\RelCatinfty[\mathrm{rel.eq}^{-1}]\to\curlyRelCatinfty$
is fully faithful, as required.
\end{proof}

\subsection{Relative $\infty$-operads and their localization}

In this subsection, we introduce the notion of relative $\infty$-operads
and their localization.
\begin{defn}
We define the ordinary category $\RelOpinfty$ of \textbf{concrete
relative $\infty$-operads} by the ordinary pullback% https://q.uiver.app/#q=WzAsNCxbMCwxLCJcXG1hdGhzZntSZWxDYXR9X3tcXGluZnR5fSJdLFsxLDEsIlxcbWF0aHNme0NhdH1fe1xcaW5mdHl9LCJdLFsxLDAsIlxcbWF0aHNme09wfV9cXGluZnR5Il0sWzAsMCwiXFxtYXRoc2Z7UmVsT3B9X3tcXGluZnR5fSJdLFswLDEsIlxcbWF0aHJte2V2fV8wIiwyXSxbMiwxLCJVIl0sWzMsMF0sWzMsMl0sWzMsMSwiXFxscmNvcm5lciIsMSx7ImxhYmVsX3Bvc2l0aW9uIjowLCJzdHlsZSI6eyJib2R5Ijp7Im5hbWUiOiJub25lIn0sImhlYWQiOnsibmFtZSI6Im5vbmUifX19XV0=
\[\begin{tikzcd}
	{\mathsf{RelOp}_{\infty}} & {\mathsf{Op}_\infty} \\
	{\mathsf{RelCat}_{\infty}} & {\mathsf{Cat}_{\infty}.}
	\arrow[from=1-1, to=1-2]
	\arrow[from=1-1, to=2-1]
	\arrow["\lrcorner"{description, pos=0}, draw=none, from=1-1, to=2-2]
	\arrow["U", from=1-2, to=2-2]
	\arrow["{\mathrm{ev}_0}"', from=2-1, to=2-2]
\end{tikzcd}\]Explicitly, an object of $\RelOpinfty$ is a pair $\pr{\mathcal{O},S}$,
where $\mathcal{O}$ is a concrete $\infty$-operad and $S$
is a set of morphisms of $U\pr{\mathcal{O}}$ which is stable under compositions, and which is stable under natural equivalences when regarded as a set of functors $[1]\to U(\mathcal{O})$ in $\curlyCatinfty$. We will regard
$\Op_{\infty}$ as a full subcategory of $\RelOp_{\infty}$ by
the functor $$\Op_\infty \hookrightarrow \RelOp_\infty, \quad \mathcal{O}\mapsto\pr{\mathcal{O}, U\pr{\mathcal{O}}\!\,^{\simeq}}$$ sending $\mathcal{O}$ to the 'minimal' relative $\infty$-operad on it.

When we work with a specific choice of $\Opinfty$, we use notations
like $\Rel\DQOp$ and terminology like ``relative dendroidal
quasioperads.''
\end{defn}

\begin{defn}
\label{def:rel_oo-opds}We define an $\infty$-category $\curlyRelOpinfty$
by the pullback% https://q.uiver.app/#q=WzAsNCxbMCwxLCJcXG1hdGhjYWx7Un1cXG1hdGhzZntlbH1cXG1hdGhjYWx7Q31cXG1hdGhzZnthdH1fe1xcaW5mdHl9Il0sWzEsMSwiXFxtYXRoY2Fse0N9XFxtYXRoc2Z7YXR9X3tcXGluZnR5fSwiXSxbMSwwLCJcXG1hdGhjYWx7T31cXG1hdGhzZntwfV9cXGluZnR5Il0sWzAsMCwiXFxtYXRoY2Fse1J9XFxtYXRoc2Z7ZWx9XFxtYXRoY2Fse099XFxtYXRoc2Z7cH1fe1xcaW5mdHl9Il0sWzAsMSwiXFxtYXRocm17ZXZ9XzAiLDJdLFsyLDEsIlUiXSxbMywwXSxbMywyXSxbMywxLCJcXGxyY29ybmVyIiwxLHsibGFiZWxfcG9zaXRpb24iOjAsInN0eWxlIjp7ImJvZHkiOnsibmFtZSI6Im5vbmUifSwiaGVhZCI6eyJuYW1lIjoibm9uZSJ9fX1dXQ==
\[\begin{tikzcd}
	{\mathcal{R}\mathsf{el}\mathcal{O}\mathsf{p}_{\infty}} & {\mathcal{O}\mathsf{p}_\infty} \\
	{\mathcal{R}\mathsf{el}\mathcal{C}\mathsf{at}_{\infty}} & {\mathcal{C}\mathsf{at}_{\infty},}
	\arrow[from=1-1, to=1-2]
	\arrow[from=1-1, to=2-1]
	\arrow["\lrcorner"{description, pos=0}, draw=none, from=1-1, to=2-2]
	\arrow["U", from=1-2, to=2-2]
	\arrow["{\mathrm{ev}_0}"', from=2-1, to=2-2]
\end{tikzcd}\]and refer to its objects as \textbf{relative $\infty$-operads}.
Concretely, a relative $\infty$-operad is an $\infty$-operad
$\mathcal{O}$ equipped with a monomorphism of $\infty$-categories
$S\hookrightarrow U\pr{\mathcal{O}}$. We often abbreviate such
pairs as $\pr{\mathcal{O},S}$.
\end{defn}

\begin{rem}
\label{rem:def:rel_oo-opds}By Remark \ref{rem:rel_oo-cats_mappingspace},
if $\pr{\mathcal{O},S}$ and $\pr{\mathcal{P},T}$ are relative
$\infty$-operads, the map
\[
\Map_{\curlyRelOpinfty}\pr{\pr{\mathcal{O},S},\pr{\mathcal{P},T}}\to\Map_{\curlyOpinfty}\pr{\mathcal{O},\mathcal{P}}
\]
is the full sub $\infty$-groupoid of $\Map_{\curlyOpinfty}\pr{\mathcal{O},\mathcal{P}}$
spanned by the maps $\mathcal{O}\to\mathcal{P}$ carrying $S$
into $T$. 

Using this, we find that $\curlyOpinfty$ embeds fully faithfully
into $\curlyRelOpinfty$ via the 'minimal' relative operad functor, just like $\Opinfty$ embeds fully
faithfully into $\RelOpinfty$.
\end{rem}

There is an obvious functor $$ \RelOpinfty\longrightarrow\curlyRelOpinfty,$$ and we say that a morphism of $\RelOpinfty$ is a \textbf{relative
	equivalence} if its image in $\curlyRelOpinfty$ is an equivalence.
We then have the following proposition, whose proof is entirely
analogous to that of Proposition \textit{\emph{\ref{prop:deloc_cat}:}}
\begin{prop}
\label{prop:deloc_op}The functor
\[
\RelOpinfty[\mathrm{rel.eq}^{-1}]\to\curlyRelOpinfty
\]
is an equivalence of $\infty$-categories, where $\mathrm{rel.eq}$ denotes the subcategory of maps whose
image in $\curlyRelOpinfty$ is an equivalence.
\end{prop}

We next turn to localization of relative $\infty$-operads. 
\begin{defn}
\label{def:loc}Let $\pr{\mathcal{O},S}$ be a relative $\infty$-operad.
We say that a map $\gamma:\mathcal{O}\to\mathcal{P}$ of $\infty$-operads
\textbf{exhibits $\mathcal{P}$ as a localization of $\mathcal{O}$
at $S$} if for every $\infty$-operad $\mathcal{Q}$, precomposition
with $\gamma$ induces an equivalence
\[
\Map_{\curlyOpinfty}\pr{\mathcal{P},\mathcal{Q}}\xrightarrow{\sim}\Map_{\curlyOpinfty}^{S}\pr{\mathcal{O},\mathcal{Q}},
\]
where $\Map_{\curlyOpinfty}^{S}\pr{\mathcal{O},\mathcal{Q}}\subset\Map_{\curlyOpinfty}\pr{\mathcal{O},\mathcal{Q}}$
denotes the full sub $\infty$-groupoid spanned by the maps $\mathcal{O}\to\mathcal{Q}$
carrying each morphism in $S$ into  $U\pr{\mathcal{Q}}\!\,^{\simeq}$.
If this condition is satisfied, we write $\mathcal{P}=\mathcal{O}[S^{-1}]$.
\end{defn}

\begin{rem}
\label{rem:loc_esssurj}In the situation described in Definition \ref{def:loc},
the underlying functor $U\pr{\gamma}\from U\pr{\mathcal{O}}\to U\pr{\mathcal{O}[S^{-1}]}$
is essentially surjective. To see this, factor $\mathcal{O}\to\mathcal{O}[S^{-1}]$
as $\mathcal{O}\xrightarrow{\alpha}\mathcal{Q}\xrightarrow{\beta}\mathcal{O}[S^{-1}]$,
where $U\pr{\alpha}$ is essentially surjective and $U\pr{\beta}$
is fully faithful. By the universal property of localizations,
we can find a dashed arrow in the diagram% https://q.uiver.app/#q=WzAsNCxbMCwwLCJcXG1hdGhjYWx7T30iXSxbMCwxLCJcXG1hdGhjYWx7T31bU157LTF9XSJdLFsxLDAsIlxcbWF0aGNhbHtRfSJdLFsxLDEsIlxcbWF0aGNhbHtPfVtTXnstMX1dIl0sWzAsMiwiXFxhbHBoYSJdLFswLDEsIlxcZ2FtbWEiLDJdLFsxLDMsIlxcbWF0aHJte2lkfV97XFxtYXRoY2Fse099W1Neey0xfV19IiwyXSxbMiwzLCJcXGJldGEiXSxbMSwyLCJcXGNoaSAiLDEseyJzdHlsZSI6eyJib2R5Ijp7Im5hbWUiOiJkYXNoZWQifX19XV0=
\[\begin{tikzcd}
	{\mathcal{O}} & {\mathcal{Q}} \\
	{\mathcal{O}[S^{-1}]} & {\mathcal{O}[S^{-1}]}
	\arrow["\alpha", from=1-1, to=1-2]
	\arrow["\gamma"', from=1-1, to=2-1]
	\arrow["\beta", from=1-2, to=2-2]
	\arrow["{\chi }"{description}, dashed, from=2-1, to=1-2]
	\arrow["{\mathrm{id}_{\mathcal{O}[S^{-1}]}}"', from=2-1, to=2-2]
\end{tikzcd}\]The commutativity of the diagram implies that $U\pr{\beta}$
is essentially surjective. Therefore, $U\pr{\gamma}=U\pr{\beta}\circ U\pr{\alpha}$
is a composite of two essentially surjective functors, concluding the claim.
\end{rem}

Note that every relative $\infty$-operad $\pr{\mathcal{O},S}$
admits a localization: For example, we can simply take $\mathcal{O}[S^{-1}]$
as the pushout % https://q.uiver.app/#q=WzAsNCxbMCwwLCJTIl0sWzAsMSwiU1tTXnstMX1dIl0sWzEsMCwiXFxtYXRoY2Fse099Il0sWzEsMSwiXFxtYXRoY2Fse099W1Neey0xfV0iXSxbMCwxXSxbMCwyXSxbMiwzXSxbMSwzXSxbMCwzLCJcXHVsY29ybmVyIiwxLHsibGFiZWxfcG9zaXRpb24iOjEwMCwic3R5bGUiOnsiYm9keSI6eyJuYW1lIjoibm9uZSJ9LCJoZWFkIjp7Im5hbWUiOiJub25lIn19fV1d
\[\begin{tikzcd}
	S & {\mathcal{O}} \\
	{S[S^{-1}]} & {\mathcal{O}[S^{-1}],}
	\arrow[from=1-1, to=1-2]
	\arrow[from=1-1, to=2-1]
	\arrow["\ulcorner"{description, pos=1}, draw=none, from=1-1, to=2-2]
	\arrow[from=1-2, to=2-2]
	\arrow[from=2-1, to=2-2]
\end{tikzcd}\]where $S[S^{-1}]$ denotes the localization of the $\infty$-category
$S$ at all of its morphisms. Combining this with Remark \ref{rem:def:rel_oo-opds},
we deduce that:

\begin{prop}
\label{prop:opr_loc}The inclusion $\curlyOpinfty\hookrightarrow\curlyRelOpinfty$
admits a left adjoint.
\end{prop}

We will refer to the left adjoint of Proposition \ref{prop:opr_loc}
as the \textbf{operadic localization functor} and denote it by
$L\from\curlyRelOpinfty\to\curlyOpinfty$.

\begin{defn}
	A morphism of relative $\infty$-operads or concrete relative $\infty$-operads is called a \textbf{local equivalence} if its image under the localization functor $L$ is an equivalence. 
\end{defn}

\begin{cor}
\label{cor:Op_infty_loc}The functors $\Op_{\infty}\hookrightarrow\RelOpinfty\to\curlyRelOpinfty\xrightarrow{L}\curlyOpinfty$
induce equivalences of $\infty$-categories 
\[
\Op_{\infty}[\mathrm{weq}^{-1}]\xrightarrow{\sim}\RelOpinfty[\mathrm{loc.eq}^{-1}]\xrightarrow{\sim}\curlyRelOpinfty[\mathrm{loc.eq}^{-1}]\xrightarrow[]{\sim}\curlyOpinfty,
\]
where $\mathrm{loc.eq}$ denotes the subcategory of local equivalences.
\end{cor}

\begin{proof}
The functor $\Opinfty[\mathrm{weq}^{-1}]\to\curlyOpinfty$ is
tautologically an equivalence, so it suffices to show that the
middle and the right arrows are equivalences. The claim on the
middle one follows from Proposition \ref{prop:deloc_op}, since $\mathrm{rel.eq}\subseteq \mathrm{loc.eq}$. The
claim on the right one follows from Proposition \ref{prop:opr_loc},
because functors admitting fully faithful right adjoints are
localizations \cite[\href{https://kerodon.net/tag/04JL}{Tag 04JL}]{kerodon}, and because $\mathrm{loc.eq}$ is by definition the subcategory of $L$-local equivalences.
\end{proof}

\begin{rem}\label{rem:operadic_loc}
	Just like the localization of $\infty$-categories, operadic localization enjoys an $(\infty,2)$-universal property. 
	More precisely, if $(\mathcal{O},S)$ is a relative $\infty$-operad, then, for any $\infty$-operad $\mathcal{P}$, the map
	\[
		\theta\from \qAlg_{\mathcal{O}[S^{-1}]}(\mathcal{P}) \to 
		\qAlg^S_{\mathcal{O}}(\mathcal{P})
	\]
is an equivalence, where the right-hand side denotes the full subcategory spanned by the maps carrying each map in $S$ to an equivalence of $U(\mathcal{P})$. 
Indeed, $\theta$ is a pullback of the map $\Fun(S[S^{-1}],U(\mathcal{P}))\to \Fun^S(S,U(\mathcal{P}))$, which is an equivalence by Remark \ref{rem:loc_cat}.
\end{rem}

\begin{rem}
	The characterization of operadic localization in Remark \ref{rem:operadic_loc} is closely related to the theory
of \textit{approximation} of $\infty$-operads, \cite[\textsection 2.3.3]{HA} (see also \cite{harpazFactorizationHomologyNotes} or \cite[\textsection C]{carmona_additivity_2025}),
which gives sufficient conditions for a map of (relative) $\infty$-operads
to be a localization. Readers should also consult \cite[Corollary B.3]{CC_notlittle}, and the related \cite[Lemma A.4.2]{karlsson_assembly_2025},
for another characterization of localization of relative $\infty$-operads
in terms of that of relative $\infty$-categories.
\end{rem}

\subsection{Proof of Theorem \ref{thm:main_intro}}

We can now state and prove the main result of this paper (Theorem
\ref{thm:main}).
\begin{defn}
A \textbf{relative operad} is a pair $\pr{\mathcal{O},S}$, where
$\mathcal{O}$ is an operad and $S$ is a replete subcategory of the
category of unary operations, whose morphisms are called \textbf{weak equivalences}. A \textbf{morphism}
of relative operads is a morphism between the underlying operads that
preserves weak equivalences. We write $\RelOp$ for the category
of relative operads and their morphisms. Note that we can identify
$\RelOp$ with a full subcategory of $\curlyRelOpinfty$.
\end{defn}

\begin{thm}
\label{thm:main}The composite $\RelOp\hookrightarrow\curlyRelOpinfty\xrightarrow{L}\curlyOpinfty$
induces an equivalence of $\infty$-categories
\[
\RelOp[\mathrm{loc.eq}^{-1}]\xrightarrow{\sim}\curlyOpinfty.
\]
\end{thm}

\begin{proof}
By Corollary \ref{cor:Op_infty_loc}, it suffices to show that
the inclusion $\iota\from\RelOp\to\RelOp_{\infty}$ induces an equivalence of $\infty$-categories
\[
\RelOp[\mathrm{loc.eq}^{-1}]\xrightarrow{\sim}\RelOpinfty[\mathrm{loc.eq}^{-1}].
\]
We prove this by taking $\Opinfty$ to be the category $\DQOp$
of dendroidal quasioperads, and then producing a functor $\Phi \from \RelOp_{\infty}\to \RelOp$ such that $\Phi \circ \iota$ and $\iota \circ \Phi$ are naturally locally equivalent to the identity functors.

%\victor{Maybe, instead of saying that there is a homotopy equivalence or a homotopy inverse, I would add a reference to Barwick--Kan where the applied criterion is proven.}
% \ken{Does it need a reference? I think it's more or less immediate from the $(\infty,2)$-universal property.}
% \victor{I agree with you this is not a deep result, but my concern is more notational than anything (we can keep the current notation in any case). The name ``homotopy equivalence" requires a choice of interval (here in $RelCat$). We are not using $[1]$ with its minimal relative cat structure, nor $(J, J=J^{\simeq})$. Instead, we use $([1],[1])$, meaning that the natural transformations are pointwise weak equivalences (in our case they are pointwise loc.eq). I think this might be confusing for someone not versed on relative stuff. That being said, BK use the name \textit{homotopy equivalence} for this hahhaha}
The key ingredient is the third author's delocalization of $\infty$-operads:
In \cite[Definition 3.11]{Pra25}, the third author defined,
for each dendroidal quasioperad $\mathcal{O}$, an ordinary operad
$\Den/\mathcal{O}$ equipped with a map
\[
\rho_{X}\from\Den/\mathcal{O}\to\mathcal{O}
\]
called the \textbf{root functor}. The root functor is natural
in $\mathcal{O}\in\DQOp$, and moreover if $\mathcal{O}$ is
cofibrant in the operadic model structure, then its image in
$\curlyOpinfty$ is a localization at the unary operations whose
images in $\mathcal{O}$ are equivalences \cite[Theorem 3.13]{Pra25}. 

With this in mind, we define a functor $\Phi\from\RelOp_{\infty}\to\RelOp$
as follows: Choose a cofibrant replacement functor $\pr -^{c}\from\dSet\to\dSet$;
thus $\pr -^{c}$ is equipped with a natural operadic equivalence
$\pr -^{c}\Rightarrow\mathrm{id}_{\dSet}$. We define $\Phi$ by the
formula
\[
\Phi\pr{\mathcal{O},S}=\pr{\Den/\mathcal{O}^{c},\rho_{\mathcal{O}^{c}}^{-1}\pr{S^{c}}},
\]
where $S^{c}$ denotes the preimage of $S$ under the map $\mathcal{O}^{c}\to\mathcal{O}$.
We claim that $\Phi$ is a homotopy inverse of $\iota$.

The root functor determines a natural transformation $\nu\from\iota\circ\Phi\Rightarrow\id_{\RelOpinfty}$
whose component at each relative operad $(\mathcal{O},S) $ is given by the composite
\[
\pr{\Den/\mathcal{O}^{c},\rho_{\mathcal{O}^{c}}^{-1}\pr{S^{c}}}\xrightarrow{\alpha}\pr{\mathcal{O}^{c},S^{c}}\xrightarrow{\beta}\pr{\mathcal{O},S}.
\]
The map $\alpha$ is a local equivalence because $\rho_{\mathcal{O}^{c}}$
is a localization, and the map $\beta$ is a relative equivalence
and hence a local equivalence. Thus $\nu$ is a natural local
equivalence. Restricting $\nu$ to $\RelOp$, we also get a natural
local equivalence $\Phi\circ\iota\Rightarrow\id_{\RelOp}$, and this completes the proof.
\end{proof}

\section{\label{sec:appl}Application: Harpaz's conjecture}

As an application of Theorem \ref{thm:main}, we give an affirmative answer to the following long-standing open conjecture:

\begin{conj}[{\cite{MO249973}}]
Lurie's \textit{operadic
	nerve functor} $N_{\Lurie}\from\Op_{\Kan}\to\LQOp$ induces an
equivalence of $\infty$-categories upon localizing at weak equivalences.
\end{conj}

The functor $N_{\Lurie}$ was defined by Lurie in \cite[Definition 2.1.1.23]{HA}. Although it is a fundamental tool to represent locally Kan simplicial operads as $\infty$-operads, to the authors' knowledge it has not been proven to induce an equivalence on homotopy theories, in contrast to its dendroidal counterparts (cf.\ Theorem \ref{thm:models_oo-operads}). A partial step in this direction has been made by Heuts-Hinich-Moerdijk in \cite[Proposition 6.1.1]{HHM:ELMDMIO}, where they prove that $N_{\Lurie}$ induces an equivalence of homotopy theories when restricted to operads \emph{without units} (also called open therein). However, the equivalence has not been proven completely, and Harpaz raised the  conjecture in the MathOverflow post \cite{MO249973}.\\

We shall now give a complete proof of the above conjecture, answering positively Harpaz's question (Theorem \ref{thm:Harpaz}). The proof is based on the following three properties of $N_{\Lurie}$:

\begin{enumerate}[label=(ON\arabic*)]

\item $N_{\Lurie}$ is a functor of relative categories.

\item For a locally Kan simplicial operad $\mathcal{O}$, the set of unary operations of $\mathcal{O}$ can be identified with those of $N_{\Lurie}(\mathcal{O})$.
	Under this identification, the assignment $(\mathcal{O},S)\mapsto (N_{\Lurie}(\mathcal{O}),S)$ determines a functor $\overline{N}_{\Lurie}\from \Rel\Op_{\Kan} \to \Rel \LQOp$ rendering the diagram
% https://q.uiver.app/#q=WzAsMyxbMSwwLCJcXG1hdGhzZntSZWxPcH0iXSxbMCwxLCJcXG1hdGhzZntSZWxPcH1fe1xcbWF0aHJte0thbn19Il0sWzIsMSwiXFxtYXRoc2Z7UmVsTFFPcH0iXSxbMCwxXSxbMCwyXSxbMSwyLCJcXG92ZXJsaW5le059X3tcXG1hdGhybXtMdXJpZX19Il1d
\[\begin{tikzcd}
	& {\mathsf{RelOp}} \\
	{\mathsf{RelOp}_{\mathrm{Kan}}} && {\mathsf{RelLQOp}}
	\arrow[from=1-2, to=2-1]
	\arrow[from=1-2, to=2-3]
	\arrow["{\overline{N}_{\mathrm{Lurie}}}", from=2-1, to=2-3]
\end{tikzcd}\]
commutative.

\item $\overline{N}_{\Lurie}$ preserves local equivalences.

\end{enumerate}

Properties (ON1) and (ON2) are immediate from the definitions.
Property (ON3) is nontrivial, and we will prove it later (Theorem
\ref{thm:ON3}). Accepting it for now, we can prove Harpaz's conjecture:
\begin{thm}
\label{thm:Harpaz}The functor $N_{\Lurie}\from\Op_{\Kan}\to\LQOp$
induces an equivalence of $\infty$-categories
\[
\Op_{\Kan}[\mathrm{weq}^{-1}]\xrightarrow{\sim}\LQOp[\mathrm{weq}^{-1}].
\]
\end{thm}

\begin{proof}
By properties (ON1) through (ON3), there is a diagram of $\infty$-categories% https://q.uiver.app/#q=WzAsNSxbMSwwLCJcXG1hdGhzZntSZWxPcH1bXFxtYXRocm17bG9jLmVxfV57LTF9XSJdLFswLDEsIlxcbWF0aHNme1JlbE9wfV97XFxtYXRocm17S2FufX1bXFxtYXRocm17bG9jLmVxfV57LTF9XSJdLFsyLDEsIlxcbWF0aHNme1JlbExRT3B9W1xcbWF0aHJte2xvYy5lcX1eey0xfV0iXSxbMCwyLCJcXG1hdGhzZntPcH1fe1xcbWF0aHJte0thbn19W1xcbWF0aHJte3dlcX1eey0xfV0iXSxbMiwyLCJcXG1hdGhzZntMUU9wfVtcXG1hdGhybXt3ZXF9XnstMX1dIl0sWzAsMSwiXFxzaW1lcSIsMl0sWzAsMiwiXFxzaW1lcSJdLFsxLDJdLFszLDEsIlxcc2ltZXEiXSxbMyw0XSxbNCwyLCJcXHNpbWVxIiwyXV0=
\[\begin{tikzcd}
	& {\mathsf{RelOp}[\mathrm{loc.eq}^{-1}]} \\
	{\mathsf{RelOp}_{\mathrm{Kan}}[\mathrm{loc.eq}^{-1}]} && {\mathsf{RelLQOp}[\mathrm{loc.eq}^{-1}]} \\
	{\mathsf{Op}_{\mathrm{Kan}}[\mathrm{weq}^{-1}]} && {\mathsf{LQOp}[\mathrm{weq}^{-1}]}
	\arrow["\sim" sloped, from=1-2, to=2-1]
	\arrow["\sim" sloped, from=1-2, to=2-3]
	\arrow[from=2-1, to=2-3]
	\arrow["\wr", from=3-1, to=2-1]
	\arrow[from=3-1, to=3-3]
	\arrow["\wr"', from=3-3, to=2-3]
\end{tikzcd}\]
which commutes up to natural equivalence. The arrows labeled
with $\sim$ are all equivalences by 
Corollary \ref{cor:Op_infty_loc} and Theorem \ref{thm:main}. Thus, the remaining arrows
must also be equivalences. In particular, the bottom one is an
equivalence.
\end{proof}
\begin{rem}
The proof of Theorem \ref{thm:Harpaz} showcases the utility
of relative operads: Since every operad can be interpreted as
an $\infty$-operad, and since localization makes sense in any
model of $\infty$-operads, relative operads serve as a very
convenient ``hub'' for comparing different models of $\infty$-operads.
We expect that our method will be useful in comparing various
models of $\infty$-operads, of which there are many.
\end{rem}

Before delving into the proof, let us state an immediate yet useful corollary. 
Let $(\delta, \lambda)$ be the adjoint equivalence between Lurie's $\infty$-operads and complete dendroidal Segal spaces constructed by Hinich-Moerdijk \cite{HM24} (see also Remark \ref{rem:symmetric monoidal structures on Op infty}).  Then, we observe that $(\delta,\lambda)$ intertwines nerve constructions:
\begin{cor} 
	There is a pair of natural equivalences  
$$
\delta \circ N_\Lurie \simeq N_d \qquad \text{and} \qquad  \lambda\circ N_{d}\simeq N_{\Lurie}.
$$
\end{cor}
\begin{proof}
	This follows from Theorem \ref{thm:Harpaz} and the fact that the $\infty$-category of $\infty$-operads only has trivial automorphisms \cite{AGG}.
\end{proof}

We now turn to the proof of (ON3). Let us recall what it says:
\begin{thm}
[Property (ON3)]\label{thm:ON3} The functor
\[
\overline{N}_{\Lurie}\from\Rel\Op_{\Kan}\to\Rel\LQOp,\,\pr{\mathcal{O},S}\mapsto\pr{N_{\Lurie}\pr{\mathcal{O}},S}
\]
preserves local equivalences. 
\end{thm}

We will prove Theorem \ref{thm:ON3} by characterizing localizations
in $\Op_{\Kan}$ and $\LQOp$ by Morita-type universal properties
(Propositions \ref{prop:char_loc_Kan} and \ref{prop:char_loc_Lurie}),
and then showing that these universal properties are equivalent.

We first characterize localizations of simplicial operads.
\begin{notation}
Let $\pr{\mathcal{O},S}$ be a relative locally Kan simplicial
operad, where $\mathcal{O}$ is $\Sigma$-cofibrant. Recall that
the category $\Alg_{\mathcal{O}}\pr{\sSet}$ of $\mathcal{O}$-algebras
in the category of simplicial sets admits a projective model
structure \cite[Theorem 13.29]{HeutsMoerdijk22}. We write $\mathscr{L}_{S}\Alg_{\mathcal{O}}\pr{\sSet}$
for its left Bousfield localization whose fibrant objects are
the $\mathcal{O}$-algebras carrying each map in $S$ to a weak
homotopy equivalence of simplicial sets. (The left Bousfield
localization may no longer be a model category, but it always
exists as a left semi-model category \cite[Theorem A]{BW24}). One way to construct $\mathscr{L}_{S}\Alg_{\mathcal{O}}\pr{\sSet}$ is to choose as localizing set for the Bousfield localization the image of $S$ through the composition
$$
\begin{tikzcd}
	U(\mathcal{O})^{\op}\ar[r] & \mathsf{Fun}(U(\mathcal{O}),\sSet)\ar[r] & \mathsf{Alg}_{\mathcal{O}}(\sSet),
\end{tikzcd} 
$$
where the first map is the coYoneda embedding, and the second one is given by operadic left Kan extension along $U(\mathcal{O})\to \mathcal{O}$ (which is a left Quillen functor by \cite[Theorem 14.43]{HeutsMoerdijk22}).
\end{notation}

\begin{prop}
\label{prop:char_loc_Kan}Let $\pr{\mathcal{O},S}$ be a relative
locally Kan simplicial operad, and let $\gamma\from\mathcal{O}\to\mathcal{P}$
be a map of locally Kan simplicial operads. Suppose that:
\begin{itemize}
\item $\gamma$ carries each map in $S$ to an equivalence;
\item $\mathcal{P}$ is $\Sigma$-cofibrant (hence so is $\mathcal{O}$).
\end{itemize}
Then the following conditions are equivalent:
\begin{enumerate}
\item The image of the map $\gamma$ in $\curlyOpinfty$ exhibits $\mathcal{P}$
as a localization of $\mathcal{O}$ at $S$.
\item The functor $U\pr{\gamma}\from U\pr{\mathcal{O}}\to U\pr{\mathcal{P}}$
is essentially surjective, and the adjunction
\[
\begin{tikzcd}
	\gamma_{!}\from\mathscr{L}_{S}\pr{\Alg_{\mathcal{O}}\pr{\sSet}} \ar[r, shift left=1.8]\ar[r,leftarrow, shift right=1.8,"\scalebox{0.8}{$\perp$}"] & \Alg_{\mathcal{P}}\pr{\mathsf{sSet}}\from\gamma^{*}
\end{tikzcd}
\]
%\[
%\gamma_{!}\from\mathscr{L}_{S}\pr{\Alg_{\mathcal{O}}\pr{\sSet}}\stackrel[\longleftarrow]{\longrightarrow}{\bot}\Alg_{\mathcal{P}}\pr{\mathsf{sSet}}\from\gamma^{*}
%\]
is a Quillen equivalence of semi-model categories. 
\end{enumerate}
\end{prop}

\begin{proof}
We start by constructing a concrete model of $\mathcal{O}[S^{-1}]$.
Let $J$ denote the groupoid generated by the poset $[1]=\{0<1\}$.
Consider the commutative diagram% https://q.uiver.app/#q=WzAsNixbMSwwLCJcXGNvcHJvZF97U31bMV0iXSxbMSwxLCJcXGNvcHJvZF97U31KIl0sWzAsMCwiQSJdLFswLDEsIkIiXSxbMiwwLCJcXG1hdGhjYWx7T30iXSxbMiwxLCJcXG1hdGhjYWx7UH0iXSxbMiwwLCJcXHNpbWVxIl0sWzIsMywiZiIsMl0sWzMsMSwiXFxzaW1lcSIsMl0sWzAsMV0sWzAsNF0sWzQsNSwiXFxnYW1tYSJdLFsxLDVdXQ==
\[\begin{tikzcd}
	A & {\coprod_{S}[1]} & {\mathcal{O}} \\
	B & {\coprod_{S}J} & {\mathcal{P},}
	\arrow["\sim", from=1-1, to=1-2]
	\arrow["f"', from=1-1, to=2-1]
	\arrow[from=1-2, to=1-3]
	\arrow[from=1-2, to=2-2]
	\arrow["\gamma", from=1-3, to=2-3]
	\arrow["\sim"', from=2-1, to=2-2]
	\arrow[from=2-2, to=2-3]
\end{tikzcd}\]where $f$ is a cofibration between cofibrant objects in the Dwyer--Kan
model structure on the category of simplicial operads, and the
arrows labeled with $\sim$ are weak equivalences in this model
structure. We factor the map $\mathcal{O}\amalg_{A}B\to\mathcal{P}$
as a trivial cofibration followed by a fibration, and write $\mathcal{O}\xrightarrow{\alpha}\mathcal{Q}\xrightarrow{\beta}\mathcal{P}$
for the resulting factorization. By \cite[Theorem 8.7]{CM13b},
$\alpha$ exhibits $\mathcal{Q}$ as a localization at $S$,
so we write $\mathcal{Q}=\mathcal{O}[S^{-1}]$. 

According to \cite[Theorem 1.5]{CG20}, condition (1) is equivalent
to the following condition:
\begin{itemize}
\item [(1$'$)]The functor $U\pr{\beta}$ is essentially surjective,
and the adjunction
\[
\begin{tikzcd}
	\beta_{!}\from\Alg_{\mathcal{O}[S^{-1}]}\pr{\sSet} \ar[r, shift left=1.8]\ar[r,leftarrow, shift right=1.8,"\scalebox{0.8}{$\perp$}"] & \Alg_{\mathcal{P}}\pr{\mathsf{sSet}}\from\beta^{*}
\end{tikzcd}
\]
is a Quillen equivalence.
\end{itemize}
Now $U\pr{\alpha}$ is essentially surjective (Remark \ref{rem:loc_esssurj}),
so $U\pr{\beta}$ is essentially surjective if and only if $U\pr{\gamma}$
has this property. Therefore, to prove the equivalence of (1$'$)
and (2), it suffices to prove the following: 
\begin{itemize}
\item [($\ast$)]The adjunction
\[
\begin{tikzcd}
	\alpha_{!}\from\mathscr{L}_{S}\Alg_{\mathcal{O}}\pr{\mathsf{sSet}} \ar[r, shift left=1.8]\ar[r,leftarrow, shift right=1.8,"\scalebox{0.8}{$\perp$}"] & \Alg_{\mathcal{O}[S^{-1}]}\pr{\sSet}\from\alpha^{*}
\end{tikzcd}
\]
is a Quillen equivalence.
\end{itemize}

To prove ($\ast$), we recall from \cite[Theorem 2.3]{Heu11}
that, for every dendroidal set $X$, the category $\dSet_{/X}$
admits a combinatorial simplicial model structure called the
\textbf{covariant model structure}, whose fibrant objects are
the dendroidal left fibrations over $X$. When $X$ is normal
(i.e., cofibrant in the operadic model structure), \cite[Theorem 2.7]{Heu11}
and \cite[Theorem 14.43]{HeutsMoerdijk22} yield a Quillen equivalence

\begin{equation}\label{str-unstr}
\begin{tikzcd}
	\mathrm{St}\from\dSet_{/N_{d}\pr{\mathcal{O}}} \ar[r, shift left=1.8]\ar[r,leftarrow, shift right=1.8,"\scalebox{0.8}{$\perp$}"] & \Alg_{\mathcal{O}}\pr{\sSet}\from\mathrm{Un}.
\end{tikzcd}
\end{equation}
This fits into a diagram% https://q.uiver.app/#q=WzAsNCxbMSwwLCJcXG1hdGhzY3J7TH1fUyhcXG9wZXJhdG9ybmFtZXtBbGd9X3tcXG1hdGhjYWx7T319KFxcbWF0aHNme3NTZXR9KSkiXSxbMCwwLCJcXG1hdGhzY3J7TH1fUyhcXG1hdGhzZntkU2V0fV97L05fZChcXG1hdGhjYWx7T30pfSkiXSxbMCwxLCJcXG1hdGhzZntkU2V0fV97L05fZChcXG1hdGhjYWx7T31bU157LTF9XSl9Il0sWzEsMSwiXFxvcGVyYXRvcm5hbWV7QWxnfV97XFxtYXRoY2Fse099W1Neey0xfV19KFxcbWF0aHNme3NTZXR9KSJdLFsxLDAsIlxcbWF0aHJte1N0fSJdLFsxLDIsIk5fZChcXGFscGhhKV8hIiwyXSxbMCwzLCJcXGFscGhhXyEiXSxbMiwzLCJcXG1hdGhybXtTdH0iLDJdXQ==
\[\begin{tikzcd}
	{\mathscr{L}_S(\mathsf{dSet}_{/N_d(\mathcal{O})})} & {\mathscr{L}_S(\operatorname{Alg}_{\mathcal{O}}(\mathsf{sSet}))} \\
	{\mathsf{dSet}_{/N_d(\mathcal{O}[S^{-1}])}} & {\operatorname{Alg}_{\mathcal{O}[S^{-1}]}(\mathsf{sSet})}
	\arrow["{\mathrm{St}}", from=1-1, to=1-2]
	\arrow["{N_d(\alpha)_!}"', from=1-1, to=2-1]
	\arrow["{\alpha_!}", from=1-2, to=2-2]
	\arrow["{\mathrm{St}}"', from=2-1, to=2-2]
\end{tikzcd}\]commuting up to natural isomorphism, where $\mathscr{L}_{S}\pr{\dSet_{/N_{d}\pr{\mathcal{O}}}}$
denotes the left Bousfield localization of the covariant model
structure whose fibrant objects are the $S$-local left fibrations
\cite[Definition 4.3]{Pra25}.
The functor $N_{d}\pr{\alpha}_{!}$ is a left Quillen equivalence
by \cite[Proposition 4.4]{Pra25}, and the bottom horizontal arrow is a left Quillen equivalence by (\ref{str-unstr}) applied to $\mathcal{O}[S^{-1}]$. It thus suffices to show that the top horizontal arrow is a Quillen equivalence.
Reasoning as in \cite[Proposition 4.9]{Pra25}, one sees that  $\mathrm{Un}$ preserves and detects fibrant $S$-locally constant algebras, i.e., local objects in $\mathscr{L}_{S}\pr{\mathsf{Alg}_{\mathcal{O}}(\sSet)}$. This follows from the equivalence between the fiber of the unstraightening of an algebra and evaluation, $ \mathrm{Un}(F)_a \simeq F(a)$. By a standard argument with (semi)model categories (see, e.g., \cite[\emph{Errata}, Lemma 3]{HeutsMoerdijk22} for the model categorical version), the Quillen equivalence \ref{str-unstr} descends to the left Bousfield localizations, concluding the proof. 

\end{proof}
Next, we turn to the characterization of localization in Lurie's
model. 
% Recall that the model structure for Lurie quasioperads
% is enriched over the Joyal model structure on the category of
% simplicial sets (see \cite[Remark B.2.5]{HA} and \cite[Proposition 3.1.5.3]{HTT}),
% and the corresponding mapping objects are denoted by $\qAlg_{\bullet}\pr{\bullet}$.
% \victor{Since we are using a similar notation above, I would suggest to denote by $\mathsf{Alg}_{\mathcal{O}}(\mathsf{sSet})$ the previous strict notion and keep $\mathrm{Alg}$ for Lurie's quasicategories of coherent algebras. They are not easily distinguishable though... Another possibility is to write something like $\mathcal{A}\mathrm{lg}$ for Lurie's}
%\ken{I now use two different notation for algebras $\Alg$ and $\qAlg$, as suggested by Victor. Thank you Victor!}
\begin{prop}
\label{prop:char_loc_Lurie}Let $\pr{\mathcal{O},S}$ be a relative
Lurie quasioperad, and let $\gamma\from\mathcal{O}\to\mathcal{P}$
be a map of Lurie quasioperads carrying each map in $S$ to an
equivalence of $U\pr{\mathcal{P}}$. The following conditions
are equivalent:
\begin{enumerate}
\item The image of the map $\gamma$ in $\curlyOpinfty$ exhibits $\mathcal{P}$
as a localization of $\mathcal{O}$ at $S$.

\item The functor $U\pr{\gamma}\from U\pr{\mathcal{O}}\to U\pr{\mathcal{P}}$
is essentially surjective, and the functor
\[
\qAlg_{\mathcal{P}}\pr{\mathcal{S}}\to\qAlg_{\mathcal{O}}^{S}\pr{\mathcal{S}}
\]
is a categorical equivalence. Here $\mathcal{S}$ denotes the
cartesian symmetric monoidal $\infty$-category of $\infty$-groupoids.
\end{enumerate}
\end{prop}

\begin{proof}
	The implication (1)$\Rightarrow$(2) follows from Remarks \ref{rem:loc_esssurj} and \ref{rem:operadic_loc}. For (2)$\Rightarrow$(1),
suppose that condition (2) is satisfied. Since $\gamma$ carries
each map in $S$ to an equivalence, it factors through $\mathcal{O}[S^{-1}]$,
say as $\mathcal{O}\xrightarrow{\alpha}\mathcal{O}[S^{-1}]\xrightarrow{\beta}\mathcal{P}$.
We wish to show that $\beta$ is an equivalence of Lurie quasioperads.
By hypothesis and Remark \ref{rem:operadic_loc}, the functor
\[
\beta^{*}\from\qAlg_{\mathcal{P}}\pr{\mathcal{S}}\to\qAlg_{\mathcal{O}[S^{-1}]}\pr{\mathcal{S}}
\]
is an equivalence. Moreover, since $U\pr{\gamma}$ is essentially
surjective, so is $U\pr{\beta}$. Therefore, Hinich--Moerdijk's
reconstruction theorem \cite[Theorem 4.1.1]{HM24} shows that
$\beta$ is an equivalence, as claimed.
\end{proof}

We can now give a proof of Theorem \ref{thm:ON3}:
\begin{proof}
[Proof of Theorem \ref{thm:ON3}]Let $\alpha\from\pr{\mathcal{O},S}\to\pr{\mathcal{P},T}$
be a local equivalence of relative locally Kan simplicial operads.
We must show that $\overline{N}_{\Lurie}\pr{\alpha}$ is an equivalence.
Using the fact that $\overline{N}_{\Lurie}$ preserves relative
equivalences, we may assume that $\mathcal{O}$ and $\mathcal{P}$
are $\Sigma$-cofibrant. By Proposition \ref{prop:char_loc_Kan},
the adjunction
\[
\begin{tikzcd}
	\alpha_!\from\mathscr{L}_{S}\pr{\Alg_{\mathcal{O}}\pr{\sSet}}\ar[r, shift left=1.8]\ar[r,leftarrow, shift right=1.8,"\scalebox{0.8}{$\perp$}"] & \mathscr{L}_{T}\pr{\Alg_{\mathcal{P}}\pr{\sSet}}\from \alpha^*
\end{tikzcd}
\]
is a Quillen equivalence. According to the rectification theorem
of algebras over $\infty$-operads (see \cite[Theorem 7.11]{PS18}
or \cite[Theorem 7.3.1]{WY24}), this is equivalent to saying
that precomposition with $\alpha$ induces an equivalence of $\infty$-categories
$\qAlg_{N_{\Lurie}\pr{\mathcal{P}}}^{T}\pr{\mathcal{S}}\xrightarrow{\sim}\qAlg_{N_{\Lurie}\pr{\mathcal{O}}}^{S}\pr{\mathcal{S}}$.
By Proposition \ref{prop:char_loc_Lurie}, this is equivalent
to saying that $\overline{N}_{\Lurie}\pr{\alpha}$ is a local equivalence,
and the proof is complete.
\end{proof}
\subsection*{Acknowledgment}
K.A.\ was supported by JSPS KAKENHI Grant
Number 24KJ\-1443. V.C.\ was partially supported by the project PID2024-157173NB-I00 funded by MCIN/AEI/10.13039/501100011033 and by FEDER, UE. F.P.\ was supported by the Marie Skłodowska-Curie grant agreement No 945332 and the ENW-XL grant.
\bibliographystyle{amsalpha}
\bibliography{refs}

\vspace{0.5cm}
%\address{K.A.: Department of Mathematics, Graduate School of Science, Kyoto University, Kyoto 606-8502, Japan.}
\address{\textsc{K.A.: Department of Mathematics, Kyoto University, Kitashirakawa-Oiwakecho, 606-8502
Kyoto (Japan)}} \newline \email{\texttt{arakawa\_kensuke.22c@st.kyoto-u.ac.jp}}

\vspace{0.2cm}
\address{\textsc{V.C.: Max Planck Institute for Mathematics in the Sciences, Inselstrasse 22, 04103 Leipzig (Germany)}} \newline \email{\texttt{vcarmonamath@gmail.com}}

\vspace{0.2cm}
\address{\textsc{F.P.: Utrecht Geometry Center, Utrecht University, Hans Freudenthalgebouw
	Budapestlaan 6,	3584 CD Utrecht (The Netherlands)}} \newline \email{\texttt{ f.pratali@uu.nl}}

\end{document}